\def\year{2016}\relax
\documentclass[letterpaper]{article}
\usepackage{aaai16}
\usepackage{times}
\usepackage{helvet}
\usepackage{courier}

\usepackage{algorithm}
\usepackage{amsmath}
\usepackage{algpseudocode}
\usepackage{graphicx} 
\usepackage{subfigure}
\usepackage{multirow}
\usepackage{multicol}
\usepackage{amsthm}
\usepackage{amsfonts}
\usepackage{amssymb}
\usepackage{thmtools}
\usepackage{appendix}

\newcommand{\Vx}{\mathbf{x}}
\newcommand{\Vb}{\mathbf{b}}
\newcommand{\Vz}{\mathbf{z}}
\newcommand{\Vy}{\mathbf{y}}
\newcommand{\Vv}{\mathbf{v}}
\newcommand{\Vs}{\mathbf{s}}

\newcommand{\MW}{\mathbf{W}}
\newcommand{\MA}{\mathbf{A}}
\newcommand{\MB}{\mathbf{B}}
\newcommand{\MH}{\mathbf{H}}

\newcommand{\MI}{\mathbf{I}}
\newcommand{\MX}{\mathbf{X}}
\newcommand{\MF}{\mathbf{F}}
\newcommand{\MG}{\mathbf{G}}
\newcommand{\MS}{\mathbf{S}}
\newcommand{\MZERO}{\mathbf{0}}

\newcommand{\SetR}{\mathbb{R}}
\newcommand{\ConeK}{\mathcal{K}}

\def\dom{\mathrm{dom}}

\newtheorem{thm}{Theorem} 

\newtheorem{remark}[thm]{Remark}
\newtheorem{ass}[thm]{Assumption}

\begin{document}
%
\title{A Scalable and Extensible Framework for Superposition-Structured Models}
\author{
Shenjian Zhao \and  Cong Xie \and Zhihua Zhang \\
Department of Computer Science and Engineering \\
Shanghai Jiao Tong University\\
\{zhao1014,xcgoner,zhihua\}@sjtu.edu.cn\\
}
\maketitle
\begin{abstract}
\begin{quote}
In many learning tasks, structural models usually lead to better 
interpretability and higher generalization performance. In recent 
years, however,   the simple structural models such as lasso are frequently proved to be
insufficient. Accordingly, there has been a lot of work on 
``superposition-structured'' models where multiple structural 
constraints are imposed. To efficiently solve these 
``superposition-structured'' statistical models, we develop a framework 
based on a proximal Newton-type method.  Employing the smoothed conic dual 
approach with the LBFGS updating formula, we propose a scalable and extensible 
proximal quasi-Newton (SEP-QN) framework. Empirical analysis on various 
datasets shows that our framework is potentially powerful, and achieves 
super-linear convergence rate for optimizing some popular 
``superposition-structured'' statistical models such as the fused sparse group 
lasso.
\end{quote}
\end{abstract}

\section{Introduction}
\noindent 
In this paper, we consider the ``superposition-structured'' 
statistical models \cite{yang2013dirty} where multiple structural constraints 
are imposed. Examples of such structural constraints include sparsity 
constraint, graph-structure,  group-structure, etc. We could leverage such 
structural constraints via specific regularization functions. Consequently, 
many problems of relevance in ``superposition-structured'' statistical learning 
can be formulated as minimizing a composite function:
\begin{equation} \label{eq:main_problem}
	\min_{\Vx \in \SetR^p} f(\Vx) \triangleq g(\Vx) + \Psi(\Vx),
\end{equation}
where $ g $ is a convex and continuously differentiable loss function, and 
$\Psi $ is a hybrid regularization, usually defined as sum of $ N $ convex 
(non-smooth) functions. More specifically,
\begin{equation*} \label{eq:nonsmooth_part}
	\Psi(\Vx) \triangleq \sum_{i=1}^{N} \psi_i(\MW_i \Vx + \Vb_i),
\end{equation*}
each $ \psi_i $ is convex but not necessarily differentiable,
$ \MW_i \in \SetR^{q_i \times p} $ and  $ \Vb_i \in \SetR^{q_i} $ are available.
For example,
$\Psi(\Vx) \triangleq \lambda_1 \| \Vx \|_1 + \lambda_2 \|\MF  \Vx \|_1 + 
\lambda_3 \sum_{j=1}^{N-2} \| \MG_j \Vx \|_2 $
defines a fused sparse group penalty \cite{zhou2012modeling} when $ \MF $ is 
the difference matrix and $ \MG_j $ indicates the group.

Indeed, there are plenty of machine learning models, which  can be cast into 
the formulation in \eqref{eq:main_problem}. 
\begin{itemize}
 \item Generalized lasso model: all generalized lasso models such as 
the fused lasso \cite{tibshirani2005sparsity}, the sparse group 
lasso \cite{simon2013sparse}, the group lasso for logistic regression 
\cite{meier2008group} can be written as the following form:
\[
	\min_{\Vx \in \SetR^p} f(\Vx) \triangleq g(\Vx) + \lambda_1 \|\MF \Vx 
\|_1 + \sum_{j=2}^{N} \lambda_j \| \MG_j \Vx \|_2.
\] 

 \item Multi-task learning: given $r$ tasks, each with sample matrix 
$\MA^{(k)} \in \SetR^{(n_k \times p)} $ ( samples in the
k-th task) and labels $\Vy^{(k)}$, \citeauthor{jalali2010dirty} 
proposed minimizing the following objective:
\begin{align} \label{eq:multitask_origin}
	\min_{\Vx \in \SetR^p} f(\Vx)  \triangleq \sum_{k=1}^r  &
l(\Vy^{(k)}, \MA^{(k)}(\MS^{(k)} + \MB^{(k)})) \nonumber \\ 
&+ \lambda_1 \| \MS \|_1 + \lambda_2 \| \MB \|_{1,\infty},
\end{align} 
where $l(\cdot)$ is the loss function and $\MS^{(k)}$ is the k-th 
column of S. Besides, more multi-task learning like the model 
in \cite{kim2010tree} also could be cast into \eqref{eq:main_problem}.

\item Gaussian graphical model with latent variables: 
\citeauthor{chandrasekaran2010latent} 
showed that the precision matrix will have 
a low rank + sparse structure when some random variables are hidden, thus the 
``superposition-structured'' model will be much helpful.

\end{itemize}

Moreover, many real-world problems benefit from 
these models such as Gene expression, time-varying network and disease 
progression. In this paper we mainly study the computational issue  of the 
model  in \eqref{eq:main_problem}.


There are some generic methods that can be used to solve these models 
theoretically. The CVX \cite{cvx} is able to solve these models, but it is not 
scalable. The Primal-Dual approach proposed by \citeauthor{combettes2012primal} (2012)
can deal with these models, but it converges slowly.  The smoothed conic dual 
(SCD) approach was studied in \cite{lan2011primal,nesterov2005smooth} and 
\citeauthor{becker2012tfocs} (2012)  could obtain $ \mathcal{O}(\frac{1}{\epsilon}) $ 
iteration-complexity, but it needs to find the minimizer related to $ g(\Vx) $ 
in each iteration. In addition, the alternating direction method of multipliers 
(ADMM) \cite{boyd2011distributed} can also be used to solve this kind of 
problems. However,  ADMM still suffers from  the same bottleneck  as the methods 
mentioned earlier. Additionally, as we know that disk I/O is the bottleneck of 
computation, so it is important to reduce the number of evaluating $ g(\Vx) $. 
In summary, it is challenging to efficiently solve the model on large-scale 
datasets.

Recently, there has been a flurry of activity about developments of Newton-type 
methods for minimizing composite functions \eqref{eq:main_problem} in the 
literature. In particular, in \cite{lee2014proximal,becker2012quasi} the authors 
focused on minimizing a composite function, which contains a convex smooth 
function and a convex non-smooth function with a simple proximal mapping. They 
also analyzed the convergence rate of various proximal Newton-type methods. 
\citeauthor{schmidt2011projected} (2011) discussed a projected quasi-Newton 
algorithm, but the sub-iteration procedure costs too much. \citeauthor{hsieh2014quic} 
(2014) further generalized the Newton method to handle some dirty statistical 
estimators. Their  
developments ``open up the state of the art but forbidding class of M-estimators 
to very large-scale problems.'' In addition, there have already been plenty of 
packages that implement these Newton-type methods such as LIBLINEAR 
\cite{fan2008liblinear}, GLMNET \cite{friedman2009glmnet,yuan2012improved},  but 
are limited to solve simple models such as lasso and elastic net.

To solve the ``superposition-structured'' models in (\ref{eq:main_problem}) on  
the large-scale problem, we resort to a proximal quasi-Newton method which 
converges superlinearly \cite{lee2014proximal}. We develop a Scalable and 
Extensible Proximal Quasi-Newton  (SEP-QN) framework to solve 
these models. More specifically, we apply a smoothed conic dual (SCD) approach 
to solving a surrogate of the original model (\ref{eq:main_problem}).
We employ the LBFGS updating formula, so that the surrogate problem could be 
solved not only efficiently but also robust. Moreover, we present several 
accelerating techniques including adaptive initial Hessian, warm-start and 
continuation SCD to solve the surrogate problem more efficiently and gain faster 
convergence rate.

In the following we start by presenting our SEP-QN framework for solving the 
``superposition-structured'' statistical models. 
Then we present the approach to solve the surrogate problem, followed
by theoretical analysis and concluding  empirical analysis.

\section{The SEP-QN Framework}
\label{sec:ace-pqn}

In this section we present the SEP-QN framework for solving the 
``superposition-structured'' statistical model in \eqref{eq:main_problem}.
We refer to $ g(\Vx) $ as ``smooth part'' and $ \Psi(\Vx) $ as ``non-smooth 
part." Usually, $ g(\Vx) $  is a loss function. For example, $ g(\Vx) 
\triangleq 
\frac{1}{n} \sum_{i=1}^n (y_i {-} {\bf a}_i^T \Vx)^2$ in the least squares 
regression problem where the ${\bf a}_i \in \SetR^p$ are input vectors and $y_i 
\in \SetR$ are the corresponding outputs, and $ g(\Vx) \triangleq \frac{1}{n} 
\sum_{i=1}^n  \log(1+ \exp(-y_i {\bf a}_i^T\Vx ))$ in the logistic regression 
where the $y_i \in \{-1, 1\}$. We are especially interested in the large-scale 
case; i.e., the number of training data $ n $ is large.

\subsection{Basic Framework}

Roughly speaking, the method is built on a line search strategy, which produces 
a sequence of points $ \{\Vx_k\} $ according to
\[
  \Vx_{k+1} = \Vx_k + t_k \Delta \Vx_k,
\]
where $ t_k $ is a step length calculated by backtrack, and $ \Delta \Vx_k $ is 
a descent direction. We compute the descent direction by minimizing a surrogate 
$ \hat{f}_{k}$  of the objective function $ f $. Given the $k$th estimate 
$\Vx_k$ of $\Vx$, we let $ \hat{f}_{k}(\Vx) $  be a local approximation of $f$ 
around $\Vx_k$. The descent direction $\Delta \Vx_k$ is  obtained by solving 
the following surrogate problem:
\begin{equation} \label{eq:proximal_local_model}
	\min_{\Vx} \; \hat{f}_{k}(\Vx).
\end{equation}
Proximal Newton-type methods approximate only the smooth part $ g $ with a 
local quadratic form. Thus, in this paper the surrogate function is  defined by
\begin{align} \label{eq:proximal_local_model_expend}
	\hat{f_k}(\Vx) =~& \hat{g}_k(\Vx) + \Psi(\Vx)  \nonumber \\
	=~&  g(\Vx_k) + \nabla g(\Vx_k)^T (\Vx - \Vx_k) \nonumber \\
	&+ \frac{1}{2}(\Vx-\Vx_k)^T \MH_k (\Vx -\Vx_k) + \Psi(\Vx), 
\end{align}
where $ \MH_k $ is a $p\times p$ positive definite matrix as 
approximation to the  Hessian of $ g$ at $\Vx=\Vx_k$. There are many 
strategies for choosing $ \MH_k $, such as BFGS and LBFGS 
\cite{nocedal2006numerical}. Considering the use in the large-scale problem, we 
will employ LBFGS to compute $\MH_k$.

After we have obtained the minimizer $ \hat{\Vx}_k $ of 
\eqref{eq:proximal_local_model}, we use the line search procedure such as 
backtracking  to select the step length $ t_k $ such that a sufficient descent 
condition is satisfied \cite{lee2012proximal}. That is,
\begin{equation} \label{eq:sufficient_descent_condition}
	f(\Vx_k + t_k \Delta \Vx_k) \leq f(\Vx_k) + \alpha t_k \gamma_k,
\end{equation}
where $\alpha \in (0, 1/2)$, $\Delta \Vx_k  \triangleq \hat{\Vx}_k - \Vx_k$, and
\[
	\gamma_k  \triangleq \nabla g(\Vx_k)^T\Delta \Vx_k  + \Psi(\Vx_k+\Delta 
\Vx_k) - \Psi(\Vx_k).
\]

Algorithm \ref{alg:EPQN} gives the basic framework  of SEP-QN. The key is 
to solve the surrogate problem \eqref{eq:proximal_local_model} when there are 
multiple structural constraints. In Algorithm~\ref{alg:sub_scd_solver} we 
present the method of solving the problem \eqref{eq:proximal_local_model}. 
Moreover, we develop several techniques to further 
accelerate our method. Specifically, we propose an acceleration schema by 
adaptively adjusting the initial Hessian $ \MH_0 $ in 
Algorithm~\ref{alg:adaptive_hessian}. We will see that  with an appropriate $ 
\MH_0 $, $ \MH_k $ can be a better approximation of $ \nabla^2 g(\Vx_k) $, 
leading to a much faster convergent procedure.

\begin{algorithm}[ht]
\caption{The SEP-QN Framework}\label{alg:EPQN}
\begin{algorithmic}[1]
\Require $ \Vx_0$ and $\MH_0$
\Ensure  $ \Vx_0 \in \dom f$, and $\MH_0 $ is a scaled identity matrix 
(positive definite).
\State $ S \gets [ ] $, $ Y \gets [ ] $, and $ \beta \gets 2 $
\Repeat
\State Update $ \MH_k $ using LBFGS, where $ \MH_k $ is symmetric positive  
definite.
\State Solve the problem in \eqref{eq:proximal_local_model} for a descent 
direction:
\begin{equation*}
	\Delta \Vx_k \gets 
	\operatorname*{argmin}_{\Delta} \hat{f}_k(\Vx_k+ \Delta) 		
        \text{\quad (Alogrithm~\ref{alg:sub_scd_solver})}
\end{equation*}
\State Search $ t_k $ with backtracking method.
\State Update $ \Vx_{k+1} \gets \Vx_k + t_k \Delta \Vx_k $
\If {$(\Vx_{k+1} - \Vx_{k})^T(\nabla g(\Vx_{k+1}) - \nabla g(\Vx_k)) > 0$}
\State $ S \gets [S, \Vx_{k+1} - \Vx_{k}]$
\State $ Y \gets [Y, \nabla g(\Vx_{k+1}) - \nabla g(\Vx_k)]  $
\begin{small}
\State $ \MH_0  \gets $ $ Ada\_Hess(t_k, \beta, \Vx_{k+1} - \Vx_{k}, \nabla	
g(\Vx_{k+1}) -  \nabla g(\Vx_k), \MH_0) $ (Algorithm~\ref{alg:adaptive_hessian})
\end{small}
\EndIf
\Until{stopping condition is satisfied}
\end{algorithmic}
\end{algorithm}

\subsection{The Solution of the Surrogate Problem 
\eqref{eq:proximal_local_model}}
If there is only one non-smooth function in $\Psi(\Vx)$ (i.e., $N$=1) with 
simple proximal mapping, we can solve the surrogate problem 
\eqref{eq:proximal_local_model} directly and efficiently via various optimal 
first-order algorithms such as FISTA \cite{beck2009fast} and  coordinate 
descent which is used in LIBLINEAR \cite{fan2008liblinear} and GLMNET 
\cite{friedman2009glmnet,yuan2012improved}. In this paper we mainly 
consider the case that  there are multiple non-smooth functions. In this case, 
we could use SCD or ADMM to solve the problem. Since we empirically observe that 
 SCD outperforms ADMM, we resort to the SCD approach.

\subsubsection{The SCD Approach}

In order to solve the problem \eqref{eq:proximal_local_model} efficiently when 
$ N > 1 $, we employ the SCD approach. The main idea is to solve the surrogate 
problem via its dual.

We first reformulate our concerned problem in \eqref{eq:proximal_local_model} 
into the following form: 
\begin{align} \label{eq:primal_objective}
& {\min} \quad  \mathcal{P}(\boldsymbol{\nu}) = \hat{g}_k(\Vx) +  
\sum_{i=1}^{N} t_i  \\
&\text{s.t. \quad} (\MW_i \Vx + \Vb_i, t_i) \in \ConeK_{\psi_i}, 
\nonumber
\end{align}
where $ \boldsymbol{\nu} = (\Vx, t_1, ..., t_N), t_i $ are new scalar 
variables, 
and $ \ConeK_{\psi_i} $ is a closed convex cone (usually the epigraph $ 
\psi_i(\MW_i \Vx + \Vb_i) \leq t_i $). Since projection onto the set $ \{ \Vx | 
(\MW_i \Vx + \Vb_i, t_i) \in \ConeK_{\psi_i} \} $ might be expensive, we 
address this issue by solving the dual problem.

We denote the dual variables by $\boldsymbol{\lambda} = (\Vz_1, \tau_1, ..., 
\Vz_N, \tau_N) $, $ \Vz = (\Vz_1,...,\Vz_N) $, where  $ (\Vz_i, \tau_i) \in 
\ConeK_{\psi_i}^{*} $. And $  \ConeK_{\psi_i}^{*} $ is the dual cone defined by
$$ \ConeK_{\psi_i}^{*} = \{ \Vx : \Vx^T \Vy \geq 0 \text{~for all } \Vy \in  
\ConeK_{\psi_i} \}.$$ Let us take an example in which $ \psi_i(\Vx) = \|\MW_i 
\Vx + \Vb_i\|_1 \leq t_i $. Then
$ \ConeK_{\psi_i} = \{ (\MW_i \Vx + \Vb_i, t_i) : \| \MW_i \Vx + \Vb_i\|_1 \leq 
t_i \} $ and
$ \ConeK_{\psi_i}^{*} = \{ (\Vz_i, \tau_i) : \| \Vz_i \|_{\infty} \leq \tau_i \} 
$.

The Lagrangian and dual functions are given by
\begin{align}
	\mathcal{L}(\boldsymbol{\nu};\boldsymbol{\lambda})
	& = \hat{g}_k(\Vx) + \sum_{i=1}^N(t_i - \Vz_i^T(\MW_i \Vx + \Vb_i) - 
\tau_i t_i), \nonumber
\end{align}
\begin{align} \label{eq:dual_problem}
	\mathcal{D}(\boldsymbol{\lambda}) & = \inf_{\Vx, t_i} \; \Big\{ 
	\mathcal{L}(\Vx,t_i;\Vz_i,\tau_i) \nonumber\\
	& \triangleq  \hat{g}_k(\Vx) + \sum_{i=1}^N(t_i - \Vz_i^T(\MW_i \Vx + 
\Vb_i) - \tau_i t_i) \Big\}.
\end{align}
The Lagrangian is unbounded unless $ \tau_i = 1 $.  Because the appropriate 
Hessian matrix is positive definite, this problem strongly convex, guaranteeing 
the convergence rate.

Denote $ \mathcal{D}^-(\Vz) = - \mathcal{D}(\boldsymbol{\lambda}) = 
-\mathcal{D}(\Vz_1,1,...,\Vz_N,1) $, and suppose $ 
\hat{\Vx}(\Vz) $ is the unique Lagrangian minimizer. \citeauthor{nesterov2005smooth} (2005) 
proved that $ \mathcal{D}^-(\Vz) $ is convex and continuously 
differentiable,  and that $ \nabla \mathcal{D}^-(\Vz) = (\MW_1 
\hat{\Vx}(\Vz) + \Vb_1, \ldots, \MW_N \hat{\Vx}(\Vz) + \Vb_N )^T$  is Lispchitz 
continuous.  Thus, provably convergent and accelerated gradient methods in the 
Nesterov style are possible. 

In particular, we need to minimize $ \mathcal{D}^-(\Vz) $. A standard 
gradient projection step for the smoothed dual problem is
\begin{equation} \label{eq:standard_project}
	\Vz^{(j+1)} = \operatorname*{arg\,min}_{\Vz:(\Vz_i, 1) \in 
	\ConeK_{\psi_i}^{*}}
	\| \Vz - \Vz^{(j)} + \delta^{(j)} \nabla \mathcal{D}^-(\Vz^{(j)}) 
	\|_2^2.
\end{equation}
Then we need to obtain $  \hat{\Vx}(\Vz^{(j)}) $ and $ \nabla 
\mathcal{D}^-(\Vz^{(j)}) $. By substituting $ \hat{g}_k(\Vx) $ into 
\eqref{eq:dual_problem},  collecting the linear and quadratic terms, and 
eliminating the unrelated terms, we get the reduced Lagrangian
\begin{small}
\[
	\hat{\mathcal{D}}(\Vz) = \inf_\Vx \; \Big\{  \frac{1}{2} \Vx^T 
\MH_k \Vx + \Vx^T (\nabla g(\Vx_k) - \MH_k \Vx_k - 
	\sum_{i=1}^N\MW_i^T \Vz_i ) \Big\}.
\]
\end{small}
The minimizer $ \hat{\Vx}(\Vz^{(j)}) $ is given by
\begin{equation} \label{eq:scd_x_update}
	\hat{\Vx}(\Vz^{(j)}) = - \MH_k^{-1} (\nabla g(\Vx_k) - \MH_k 
	\Vx_k - \sum_{i=1}^N \MW_i^T \Vz_i^{(j)}).
\end{equation}
From \eqref{eq:dual_problem}, \eqref{eq:standard_project} and 
\eqref{eq:scd_x_update}, the minimization problem over $ \Vz $ is separable, so 
it can be implemented in  parallel. The solution is given by
\begin{small}
\begin{equation} \label{eq:scd_z_update}
	\Vz_i^{(j+1)} = \operatorname*{arg\,min}_{\Vz_i:(\Vz_i, 1) \in 
	\ConeK_{\psi_i}^{*}} \frac{1}{2\delta^{(j)}}\|\Vz_i-\Vz_i^{(j)}\|_2^2
	+ \Vz_i^T (\MW_i\hat{\Vx}(\Vz^{(j)}) + \Vb_i).
\end{equation}
\end{small}
From \eqref{eq:scd_x_update} and \eqref{eq:scd_z_update}, we obtain the 
specific 
AT method \cite{auslender2006interior}  to solve the problem 
\eqref{eq:proximal_local_model} in Algorithm \ref{alg:sub_scd_solver}.

\begin{algorithm}[ht]
\caption{Solve Problem \eqref{eq:proximal_local_model} via 
SCD}\label{alg:sub_scd_solver}
\begin{algorithmic}[1]
\Require $ \Vx_k, S, Y, \MH_0, \nabla g(\Vx_k), \Vz_i^{(0)}, \Vx_0 $
\State $ \theta^{(0)} \gets 1, \Vv_i^{(0)} \gets \Vz_i^{(0)}, j \gets 0 $
\Repeat
  \State $ \Vy_i^{(j)} \gets (1-\theta^{(j)})\Vv_i^{(j)} + \theta^{(j)} 
\Vz_i^{(j)} $
\State \label{key_update_x}
$\hat{\Vx} \gets -\MH_k^{-1} (\nabla g(\Vx_k) - \MH_k \Vx_k 
- \sum_{i=1}^N \MW_i^T \Vy_i^{(j)} ) $ by LBFGS method. 
\For{$ i \gets 1, N $}
\State 
{\small
   $ \Vz_i^{(j+1)} \gets \underset{\Vz_i:(\Vz_i, 1) \in 
\ConeK_{\psi_i}^{*}}{\text{argmin}}
  \frac{\theta^{(j)}}{2\delta^{(j)}}\|\Vz_i-\Vz_i^{(j)}\|_2^2 + 
\Vz_i^T(\MW_i\hat{\Vx}+\Vb_i)  $}
  \State  $ \Vv_{i}^{(j+1)} \gets (1-\theta^{(j)})\Vv_{i}^{(j)} + \theta^{(j)} 
\Vz_i^{(j+1)}  \nonumber $
\EndFor
 \State $ \theta^{(j+1)} \gets 2/(1+(1+4/(\theta^{(j)})^2)^{\frac{1}{2}}) $
 \State $ j \gets j + 1 $
\Until{some stopping condition is satisfied}
\State $ \Delta \gets \hat{\Vx} - \Vx_k $
\State \textbf{return} $ \Delta $
\end{algorithmic}
\end{algorithm}

There are many variants of  optimal first-order methods 
\cite{lan2011primal,beck2009fast,nesterov2007gradient,tseng2008accelerated}. 
Algorithm \ref{alg:sub_scd_solver} is a generic algorithm but may not be the 
best choice for every model.  By using the continuation 
techniques\cite{becker2012tfocs}, we could obtain the exact solution very 
quickly.


\subsection{Acceleration}
\label{sec:accelerate}

We further employ several acceleration techniques in our implementation. By 
applying these techniques we achieve much faster convergence rate which is 
comparable to the conventional proximal Newton method. Our accelerated 
implementation behaves much better than the original proximal quasi-Newton 
method in various aspects.

\subsubsection{Adaptive Initial Hessian}

LBFGS sets the initial Hessian $ \MH_0 $ as $ \frac{\Vy_k^T \Vy_k}{\Vs_k^T 
\Vy_k} \MI $. However, we find that this setting results in a much slower 
convergence procedure than the proximal Newton method. Thus, it is desirable to 
give a better initial Hessian $ \MH_0 $,  which in turn yields a better 
approximation of $ \nabla^2 g(\Vx_k)$.

\begin{restatable}{thm}{stepsufficient}
\label{th:unit_step_sufficient}
If $ (1-\alpha) \MH_k \succcurlyeq \nabla^2 g(\Vx_k) $ ~for~ $\alpha \in
(0, \frac{1}{2})$, $ \MH_k \succcurlyeq m\MI, (m > 0) $ ~and~  $ \nabla^2  g $ 
is Lipschitz continuous with constant $ L_2 $, then the unit step length 
satisfies
the sufficient decrease condition \eqref{eq:sufficient_descent_condition} after
sufficiently many iterations.
\end{restatable}

\begin{restatable}{thm}{largeH}
\label{th:larger_H0}
Assume $ \MH_k^a$ and $\MH_k^b $ are generated by the same procedure
$ \{ (\Vs_k,  \Vy_k): \Vs_k^T \Vy_k > 0\}$ but with different initial
Hessians $ \MH_0^a$ and  $\MH_0^b $, respectively.  If ~$ \MH_0^a \succ  
\MH_0^b 
\succ \MZERO $, then $ \MH_k^a \succ \MH_k^b \succ \MZERO$.
\end{restatable}
Based on Theorems \ref{th:unit_step_sufficient} and \ref{th:larger_H0}, we can 
decrease $ \MH_0 $ more aggressively. Once the unit step fails, we know that $  
(1-\alpha) \MH_k \succ \nabla^2 g(\Vx_k) $ is broken; hence we need to increase 
$ \MH_0 $. We propose our adaptive initial Hessian strategy in Algorithm 
\ref{alg:adaptive_hessian}. In practice, we set $ \alpha $ to a small number 
like $0.0001$.

\begin{algorithm}[ht]
\caption{Adaptive Initial Hessian}\label{alg:adaptive_hessian}
\begin{algorithmic}[1]
\Procedure{Ada\_Hess}{$t_k, \beta, \Vs_k, \Vy_k, \MH_0$}
  \If { $ t_k < 1 $}
    \State $ \MH_0 = \MH_0 / t_k $
    \State $ \beta = \frac{2}{1+1/\beta} $
  \EndIf	
  \State $ \MH_0 = \text{elementwise\_min}( \frac{\MH_0}{\beta}, \frac{\Vy_k^T 
\Vy_k}{\Vy_k^T 
\Vs_k} \MI)$
  \State \textbf{return} $ \MH_0 $
\EndProcedure
\end{algorithmic}
\end{algorithm}

\subsubsection{Warm start and continuation SCD}

We use the optimal dual value $\Vz_k^* $ which is obtained in solving dual of $ 
\hat{f}_k(\Vx) $ as the initial dual value to solve dual of $ 
\hat{f}_{k+1}(\Vx) $. This leads to a warm start in  solving the problem 
\eqref{eq:proximal_local_model}, and the iteration complexity will be 
dramatically reduced.

By employing continuation SCD to solve the 
problem \eqref{eq:proximal_local_model},  
the dual of the original problem \eqref{eq:primal_objective} could reach $ 
\epsilon $ optimal within $ \mathcal{O}(\sqrt{\frac{1}{\lambda_{min} \epsilon}} 
\| \Vz_k^* - \Vz_{k+1}^* \|_2)$ iterations which shows in  
\cite{nesterov2005smooth}.

\section{Theoretical Analysis} \label{sec:thm_analysis}

In this section we conduct analysis about the convergence rate of SEP-QN 
method. Because of space limitations, we give the detailed proofs in the 
supplementary. In order to provide the global convergence and solve the
problem efficiently, we make the following assumptions:
\begin{ass} \label{as:exist_solution}
$ f$ is a closed convex function and $ \underset{ \Vx }{\inf} \{f(\Vx) | \Vx 
\in 
\mathrm{dom} f\} $ is attained at some $ \Vx^* $.
\end{ass}
\begin{ass}
The smooth part $ g $ is a closed, proper convex, continuously differentiable 
function, and its gradient $ \nabla g $ is Lipschitz continuous with $ L_1 $. 
\end{ass}
\begin{ass}
The non-smooth part $ \Psi$ should be closed, proper, and convex. The 
projection 
onto the dual cone associated with each $ \psi_i $ is tractable, or 
equivalently, easy to solve problem \eqref{eq:scd_z_update}.
\end{ass}

First, we analyze the global convergence behavior of SEP-QN under these 
assumptions.  

\begin{restatable}{thm}{globalc}
\label{th:gloal_c}
If the problem \eqref{eq:proximal_local_model} is solved by continuation SCD, 
then $ \{ \Vx_k \} $ generated by the SEP-QN method converges to an optimal 
solution $ \Vx^*$ starting at any $ \Vx_0 \in \dom f$.
\end{restatable}

Under the stronger assumptions, we could derive the local superlinear 
convergence rate as shown in the following theorem. 
\begin{restatable}{thm}{localc}
\label{th:local_c}
Suppose $ g $ is twice-continuously differentiable and strongly convex with 
constant $ l $, and $ \nabla^2 g $ is Lipschitz continuous with constant $ L_2 
$. If~ $ \Vx_0 $ is sufficiently close to $ \Vx^* $, the sequence $ \{ \MH_k \} 
$ satisfies the Dennis-More criterion, and $ l \MI \preceq \MH_k \preceq L\MI $ 
~for some~ $ 0 < l \leq L $, then SEP-QN with the continuation SCD converges 
superlinearly after sufficiently many iterations.
\end{restatable}
\begin{remark} \label{re:complexity}
Suppose SEP-QN converges within $T$ iterations. If the dataset is dense, then 
the complexity of SEP-QN is $ T\mathcal{O}(np) + 
T\mathcal{O}(\frac{1}{\epsilon_s}(Mp + \sum_{i=1}^N q_i p)) $; if the dataset 
is 
sparse, the complexity is
$ T\mathcal{O}(\textnormal{nnz}) + T\mathcal{O}(\frac{1}{\epsilon_s}(Mp + 
\sum_{i=1}^N q_i p)) $, where  $M$ is the history size of LBFGS, $ 
\epsilon_s $ is the tolerance of the problem \eqref{eq:proximal_local_model} 
and 
$ \textnormal{nnz} $ is the amount of non-zero entries in the sparse dataset.
\end{remark}

We require that $ n $ or $ \textnormal{nnz} $ are relatively large, otherwise 
the complexity of the problem \eqref{eq:proximal_local_model} will go over the 
complexity of evaluating the loss function. In this case,  it would be better 
to use some first-order methods instead of SEP-QN. If ignoring the impact of $ T 
$ and the dataset is dense, the convergence time of SEP-QN is linear with 
respect to the number of features, the amount of data size, and the number of 
non-smooth terms. We will empirically validate the scalability and extensibility 
of SEP-QN in the following section.

\section{Empirical Analysis}
\begin{figure*}[htbp]
\centering
\subfigure[epochs of $\ell_1$-logistic regression]
{\label{sfig:epochs_log}
\includegraphics[width=0.31\linewidth]{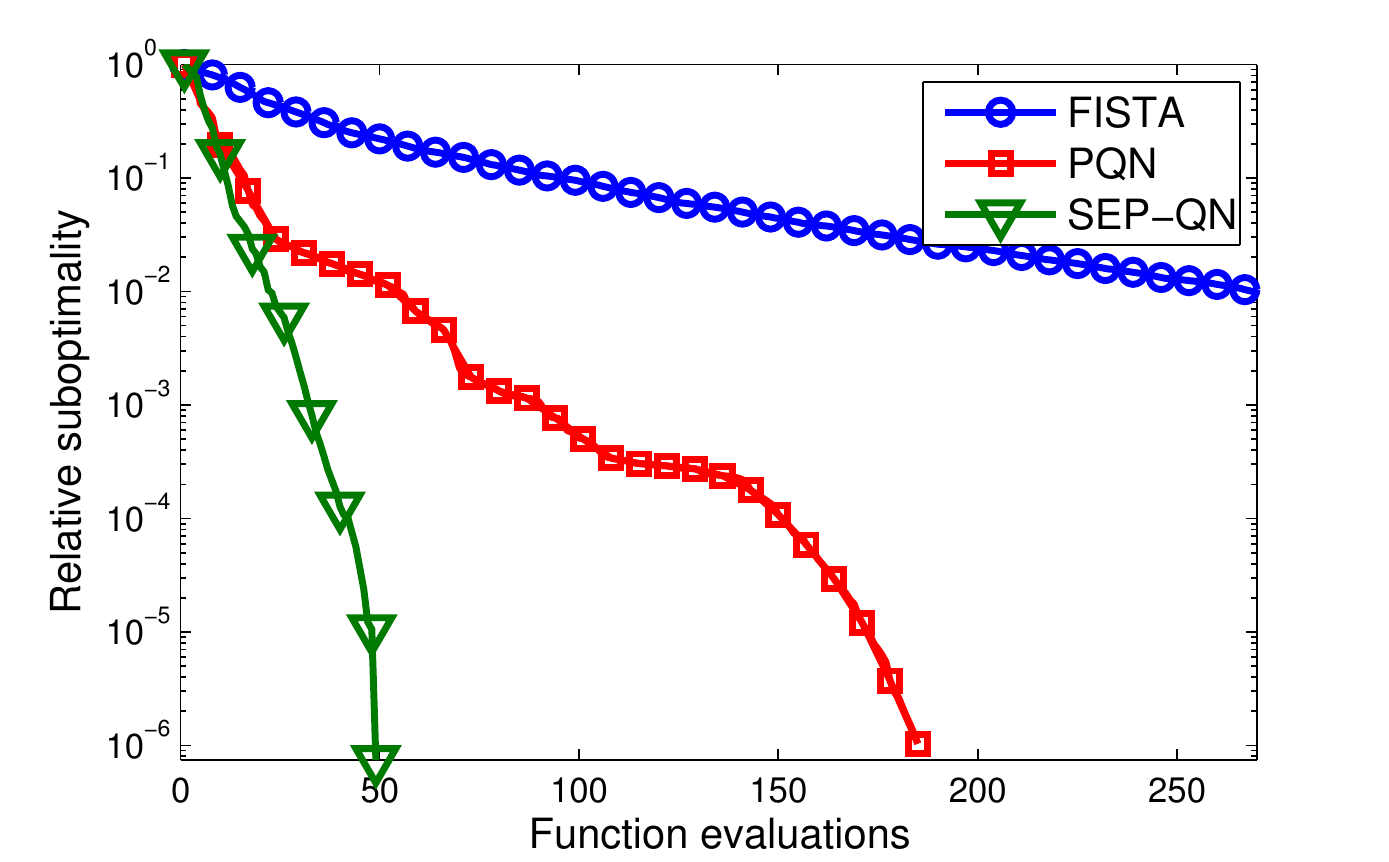}}
  \subfigure[runtime of $\ell_1$-logistic regression]
  { \label{sfig:time_log} 
\includegraphics[width=0.31\linewidth]{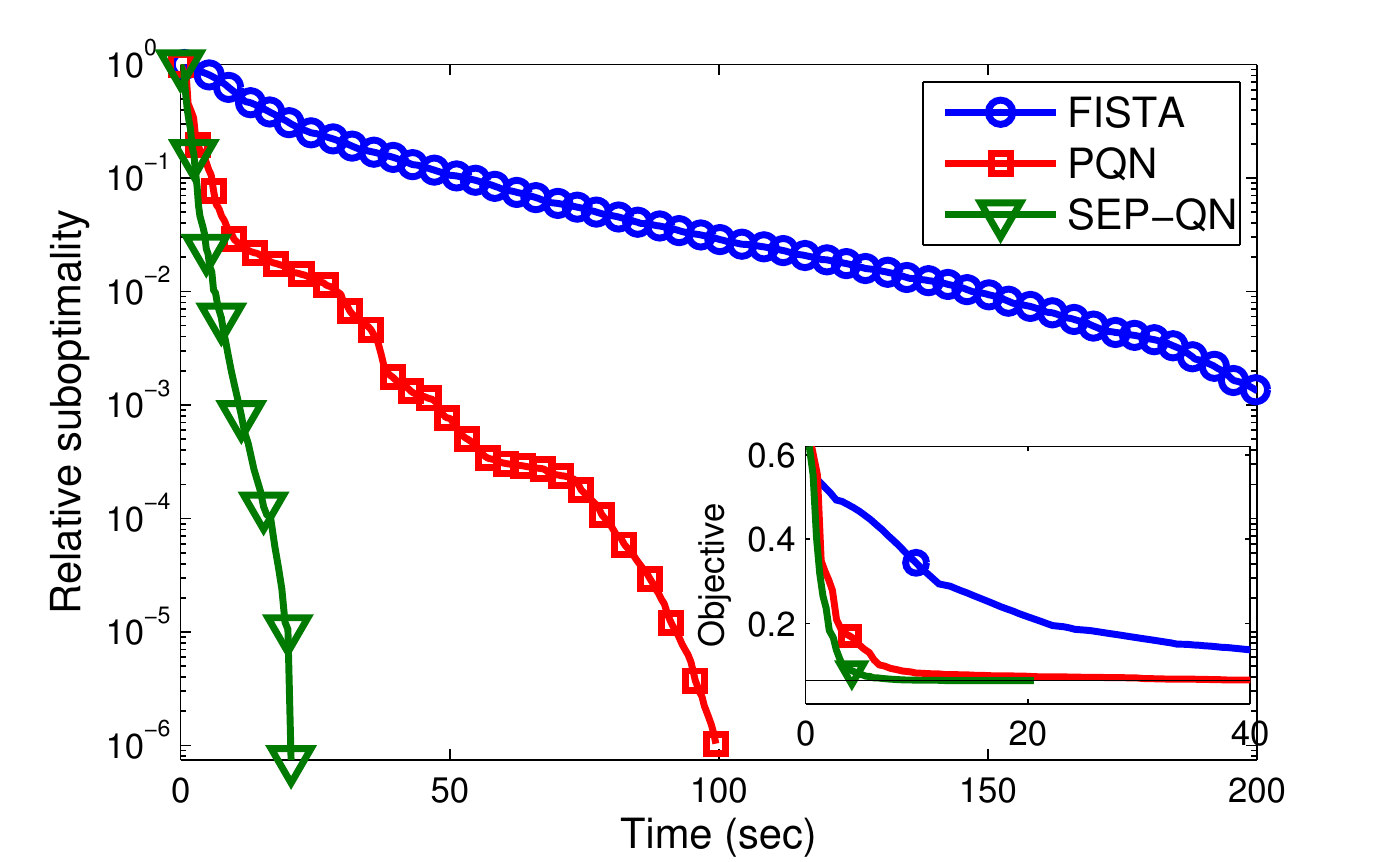}}
  \subfigure[fused $\ell_1$-logistic regression]
  { \label{sfig:fused_log} 
\includegraphics[width=0.31\linewidth]{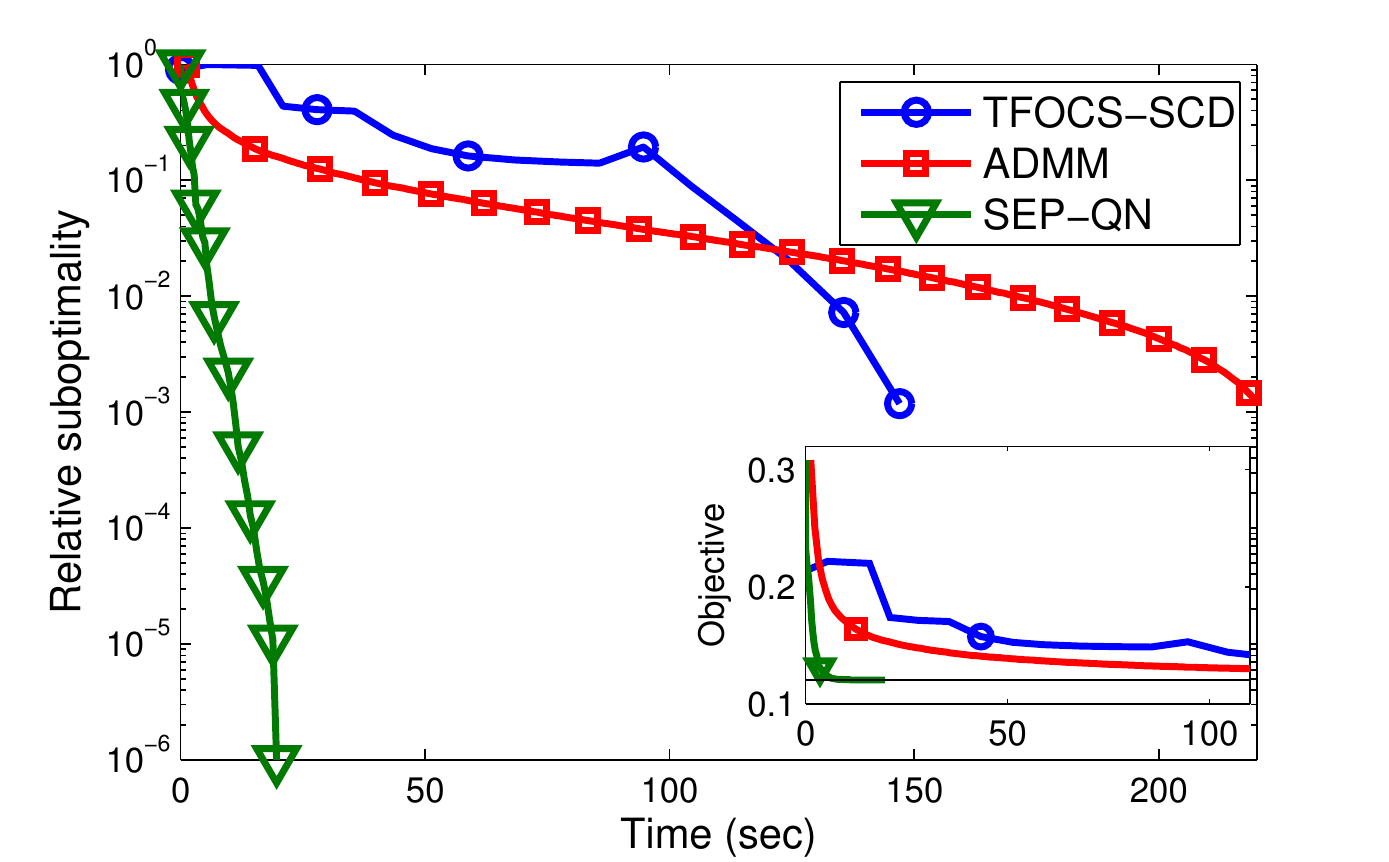}}
\caption{Convergence comparison}
\label{fig:convergence_compare}
\end{figure*}
\label{sec:exper}
We implement all the experiments on a single machine running the 64-bit version 
of Linux with an Intel Core i5-3470 CPU and 8 GB RAM. We test the SEP-QN 
framework on various real-world datasets such as \texttt{gisette} ($n= 6,000$ 
and $p= 5,000 $) and \texttt{epsilon}  ($n= 300,000$ and $p= 2,000 $) which can 
be downloaded from LIBSVM 
website\footnote{http://www.csie.ntu.edu.tw/~cjlin/libsvmtools/datasets}. 
The dataset characteristics are provided in the Table  \ref{tb:datasets}.

\begin{small}
\begin{table}[ht]
\caption{Details of the datasets in our experiments}
\label{tb:datasets} 
\begin{center}
\begin{tabular}{|c|r|r|r|r|r|}
 \hline
Dataset & p \text{\quad} & n (train) & n (test) & nnz (train) \\ \hline
epsilon     & 2,000 	& 300,000 	& 100,000 	& 600,000,000    \\
gisette     & 5,000 	& 6,000 	& 1,000 	& 29,729,997 	 \\
usps        & 649 	& 1000    	& 1000     	& 649,000 	 \\
\hline
\end{tabular}
\end{center}
\end{table}
\end{small}
\begin{figure}[ht]
\begin{center}
\includegraphics[width=0.99\linewidth]{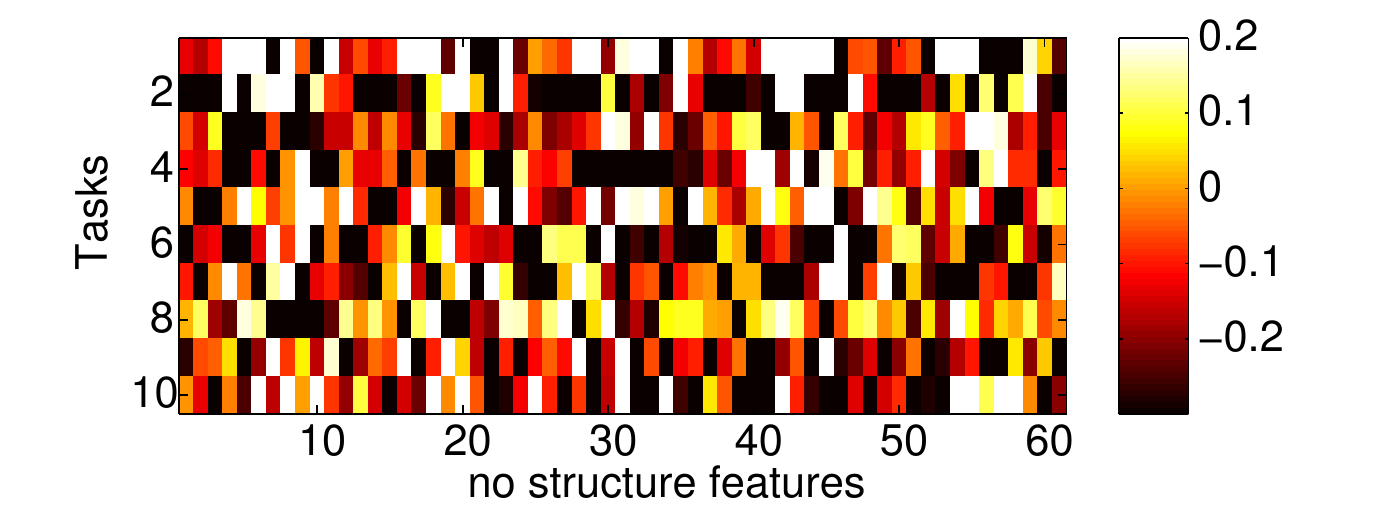}
\includegraphics[width=0.99\linewidth]{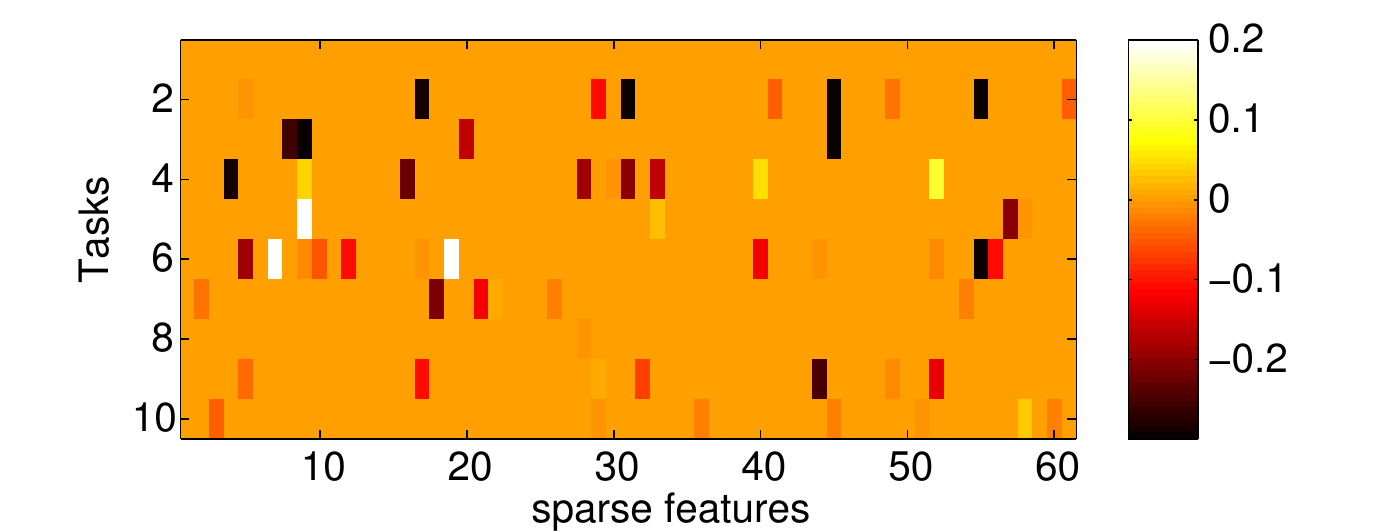}
\includegraphics[width=0.99\linewidth]{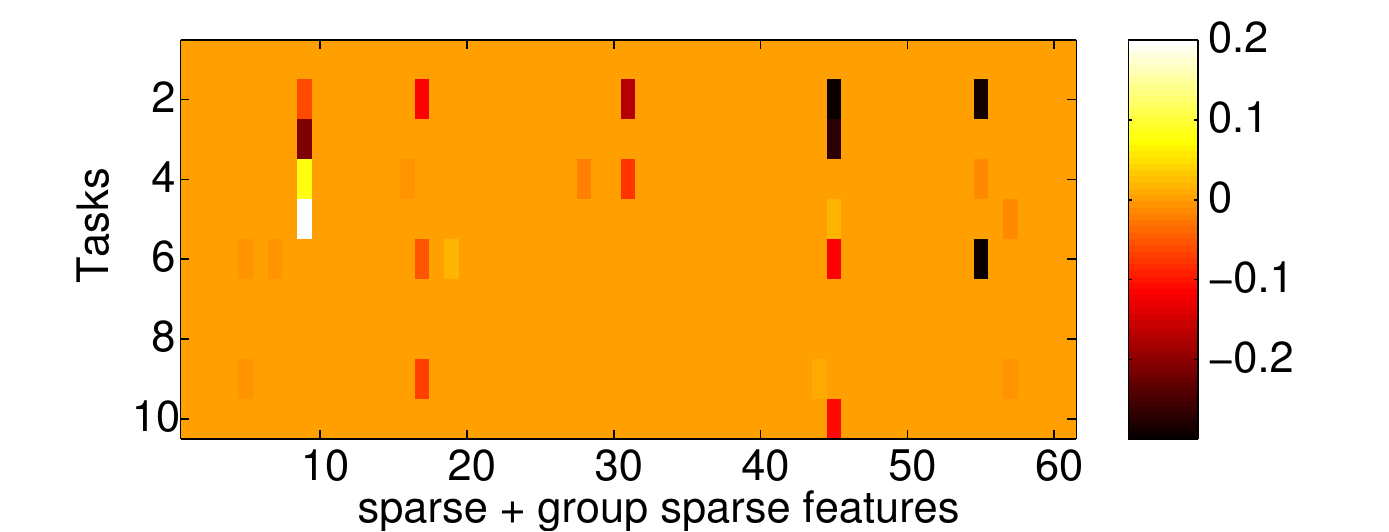}
\end{center}
\caption{Feature visualization. As shown in the colorbar, orange color 
indicates that the value of corresponding feature is 0.}
\label{fig:feature_visualization}
\end{figure}

\subsection{``Superposition-structured'' Logistic Regression}

We consider the ``superposition-structured'' logistic regression problem:
\[
 \underset{\Vx \in \SetR^p } {\min} \; \frac{1}{n} \sum_{i=1}^n {\log}(1+ 
{\exp}(-y_i {\bf a}_i^T \Vx))
 + \lambda \| \Vx \|_1 + \gamma \|\MW \Vx \|_q.
\]

We first set $\MW = \mathbf{0} $ for comparison with the result of PQN in 
\cite{lee2012proximal}. For fairness of comparison, we use the same dataset 
\texttt{gisette}  and the same setting of the tuning parameter  $ \lambda $  as 
\cite{lee2012proximal,yuan2012improved}.
The results are shown in the Figures \ref{sfig:epochs_log} and 
\ref{sfig:time_log}. We can see that the SEP-QN method has the fastest 
convergence rate,  which agrees with Theorems \ref{th:gloal_c} and 
\ref{th:local_c}.

In order to verify the effectiveness and efficiency when the model has multiple
structural constraints.We compare SEP-QN with ADMM and the direct SCD in TFOCS 
\cite{becker2012tfocs} on the fused sparse logistic regression by setting $ \| 
\MW \Vx \|_q = \| \Vx \|_{TV} $ and $ \gamma = \lambda $. Figure 
\ref{sfig:fused_log} shows that the three algorithms converge to the same 
optimal value, but SEP-QN performs much better.

\begin{table*}[htbp]
\centering
\caption{The comparisons on multi-task problems.}
\label{tb:multitask} 
\begin{tabular}{|c|l|c|c|c|c|c|}
\hline
\multirow{2}{*}{n}   & 
\multicolumn{1}{c|}{\multirow{2}{*}{\begin{tabular}[c]{@{}c@{}}relative\\  
error\end{tabular}}} & \multicolumn{3}{c|}{\begin{tabular}[c]{@{}c@{}}sparse + 
group sparse\\ (test error rate / training time )\end{tabular}} & 
\multicolumn{2}{c|}{Other Models}               \\ \cline{3-7} 
                     & \multicolumn{1}{c|}{}                                    
 
                                      & SEP-QN                                & 
QUIC \& DIRTY                         & ADMM                                  & 
Lasso                  & Group Lasso            \\ \hline
\multirow{2}{*}{100} &                                                          
 
                        $10^{-1}$       & 7.3\% / 0.32s                    
     & 8.3\% / 0.42s                         & 8.3\% / 1.5s                     
     & \multirow{2}{*}{7.9\%} & \multirow{2}{*}{7.4\%} \\ \cline{2-5}
                    &                                                  
                       $10^{-4}$                         & 6.4\% / 0.93s        
 
                & 7.4\% / 0.75s                         & 7.5\% / 4.3s          
 
               &                        &                        \\ \hline
\multirow{2}{*}{400} &                                                          
 
                           $10^{-1}$            & 3.0\% / 1.2s                  
        & 2.9\% / 1.01s                         & 3.0\% / 3.6s                  
 
       & \multirow{2}{*}{3.0\%} & \multirow{2}{*}{3.1\%} \\ \cline{2-5}
                     &                                                          
 
                           $10^{-4}$            & 2.6\% / 2.0s                  
 
       & 2.5\% / 1.55s                         & 2.6\% / 11.0s                  
 
      &                        &                        \\ \hline
\end{tabular}
\end{table*} 
\begin{figure*}[htbp]
\begin{center}
\begin{subfigure}
  [feature-number scalability]
  { \label{sfig:feature} 
\includegraphics[width=0.32\linewidth]{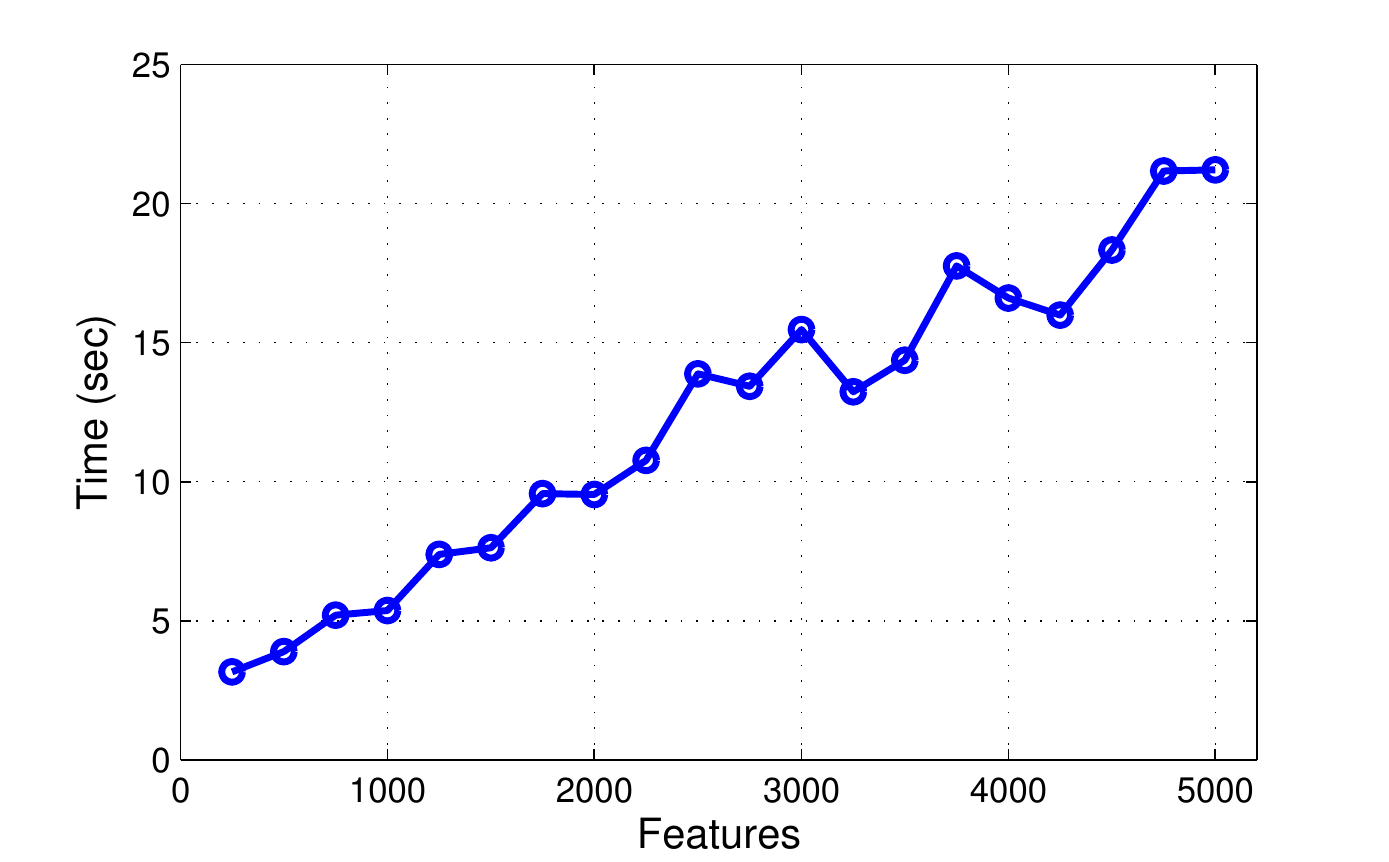}}
\end{subfigure}%
\begin{subfigure} 
  [data-size scalability]
  { \label{sfig:datasize} 
\includegraphics[width=0.32\linewidth]{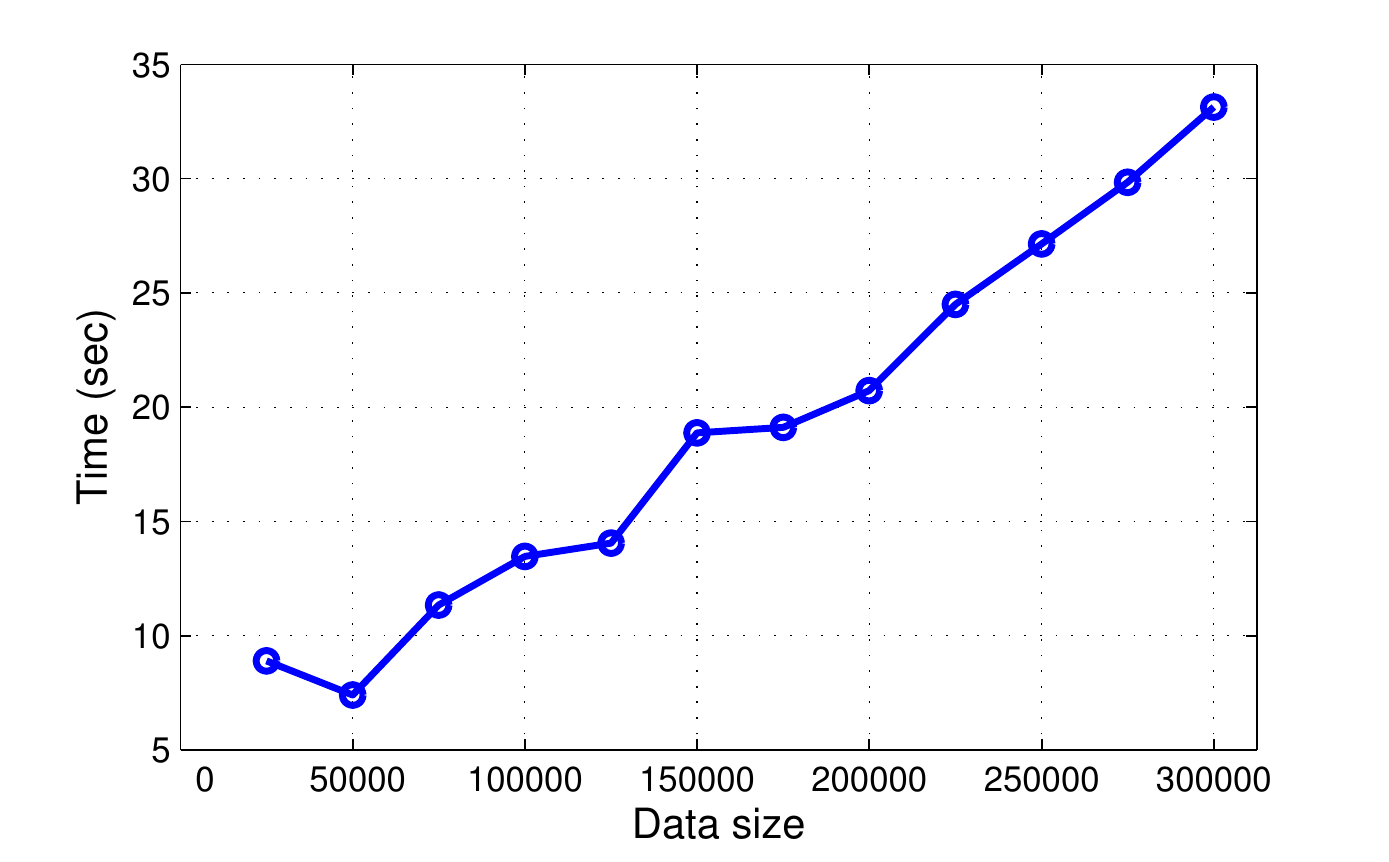}}
\end{subfigure}%
\begin{subfigure}
  [nonsmooth-terms extensibility]
  { \label{sfig:nonsmooth} 
\includegraphics[width=0.32\linewidth]{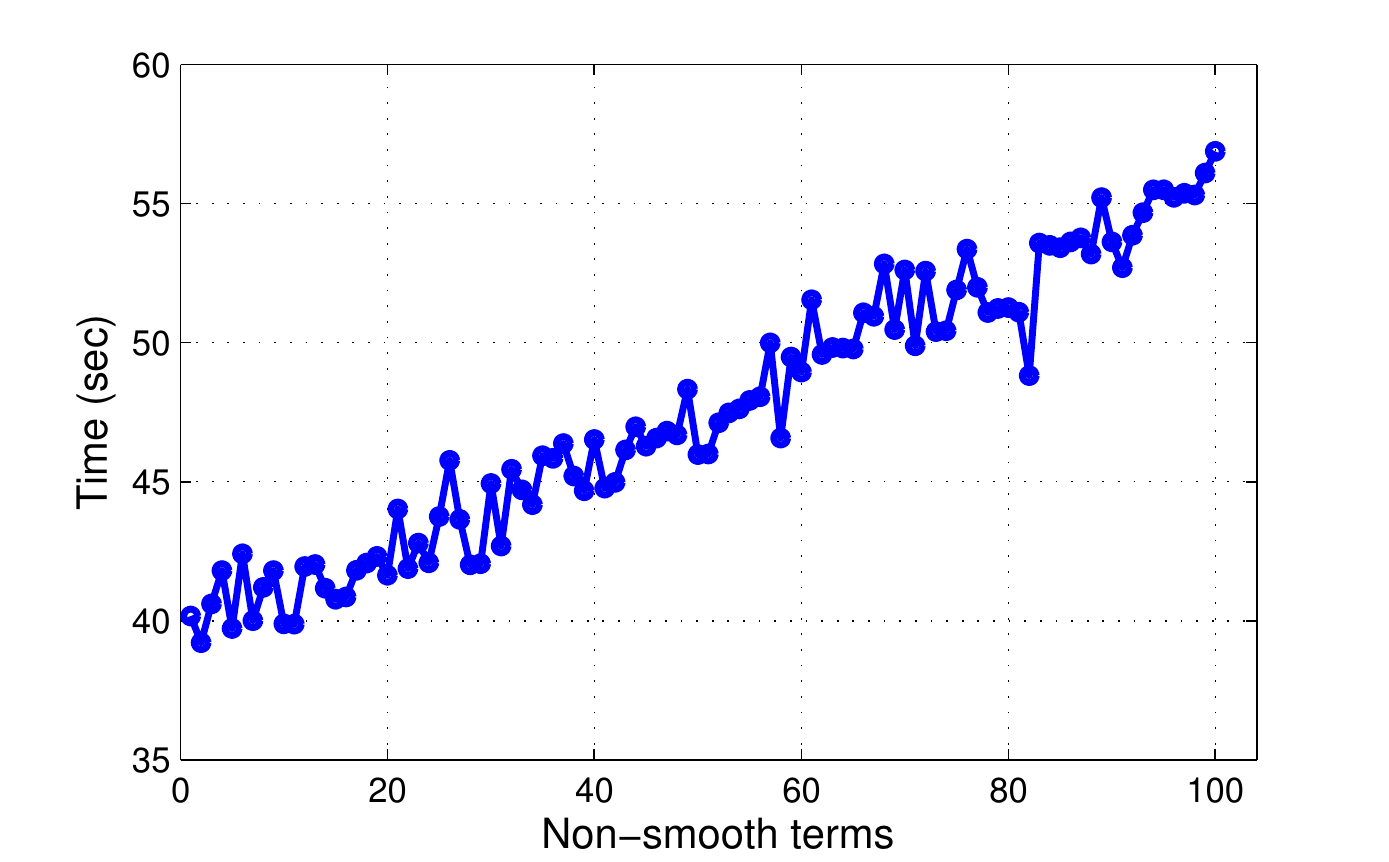}}
\end{subfigure}%
\end{center}
\caption{Scalability and extensibility}
\label{fig:scalability}
\end{figure*}

\subsection{Multi-task Learning}
Next we solve the multi-task learning problem  where the parameter matrix $ 
\MX $ will have a sparse + group sparse structure. In our framework, there is 
no need to seperate the parameter matrix $\MX$ into $ \MS + \MB $ as in 
\cite{jalali2010dirty,hsieh2014quic}. Instead of using the 
square loss (as in \cite{jalali2010dirty}), we consider the logistic loss, 
which gives better performance. Thus, $ \MX $ could be estimated by the following 
objective function,
\begin{small}
\begin{align*} 
	\min_{\MX \in \SetR^{p \times r}} \sum_{k=1}^r 
l_{logistic}(\Vy^{(k)}, \MA^{(k)}\MX^{(k)} ) 
+ \lambda \| \MX \|_1 + \gamma \| \MX \|_{1,2}.
\end{align*}
\end{small}
We follow \cite{jalali2010dirty,hsieh2014quic} 
and transform multi-class problems into multi-task problems. For fairness 
of comparison, we test on the same dataset USPS which was first collected in 
\cite{van1998handwritten} and subsequently widely used in multi-task papers as a 
reliable dataset for handwritten recognition algorithms. There are  $r = 10$ 
tasks, and each handwritten sample consists of $p = 649$ features.
In \cite{jalali2010dirty,hsieh2014quic}, the authors demonstrated that on 
USPS, using 
sparse and group sparse regularizations together outperforms the models with a 
single regularizer.

We  visualize features that estimated by our SEP-QN framework in Figure 
\ref{fig:feature_visualization}, and we just plot the first sixty features to 
provide a clear visualization. Figure \ref{fig:feature_visualization} shows that the 
feature structure is well maintained by the regularizer. The promising results of 
``sparse + group sparse structure" further validate the effectiveness of our 
SEP-QN framework. As shown in Table \ref{tb:multitask}, our  SEP-QN framework 
is comparable to QUIC \& DIRTY which is the state-of-art method.
Unlike QUIS \& DIRTY, our implementation is straightforward in the SEP-QN 
framework. Because of broad interest of our framework, it may be slower than
QUIC \& DIRTY on some specifical datasets. However, we will show that our framework is 
scalable by the experiments in the following section.

\subsection{Scalability and Extensibility}

We consider the group generalized lasso problem
\cite{tibshirani2005sparsity,simon2013sparse,meier2008group}, 
but use the logistic loss function  instead for classification. Specifically,
\begin{align}
\underset{\Vx \in \SetR^p } {\min} & \;
\frac{1}{n} \sum_{i=1}^n {\log}(1+ {\exp}(-y_i {\bf a}_i^T \Vx))  \nonumber \\
& + \lambda_1 \|\Vx \|_1 + \lambda_2 \|\MF \Vx \|_1 + \sum_{j=1}^{N-2} \gamma_j 
\|\MG_j \Vx \|_2. \nonumber
\end{align}
As far as we know, there is no an efficient algorithm to solve this model. Note 
that this dirty model may not be a good choice for  \texttt{gisette} and 
\texttt{epsilon} datasets.  We just use this model to validate the scalability 
and extensibility of SEP-QN framework. 

We use the fused sparse logistic regression ($ \lambda_1 = \frac{2}{n}, 
\lambda_2 = \frac{2}{n}, N = 2  $) and \texttt{gisette} dataset to test the 
feature-number scalability of SEP-QN as shown in Figure \ref{sfig:feature}. 
Then we test the data-size scalability on \texttt{epsilon} dataset as shown in 
Figure \ref{sfig:datasize}. We can see that the convergence time is linear with 
respect to the number of features as well as the amount of data.

Then we use the group sparse logistic model ($ \lambda_1 = \frac{2}{n}, 
\lambda_2 = 0, \gamma_j = \frac{2}{n}  $) and \texttt{epsilon} dataset to test 
the nonsmooth-terms extensibility of SEP-QN. As shown in Figure 
\ref{sfig:nonsmooth}, the convergence time is linear with respect to the number 
of non-smooth terms. These experiments further verify  Remark 
\ref{re:complexity} under the assumptions.

\section{Conclusion}
\label{sec:conclusion}

In this paper, we have generalized  the proximal quasi-Newton method to handle 
``superposition-structured'' statistical models and devised a SEP-QN framework.
With the help of the SCD approach and  LBFGS updating formula, we can 
solve the surrogate problem in an efficient and feasible way. We have explored 
the global convergence and the super-linear convergence both theoretically 
and empirically. Compared with prior methods, SEP-QN converges significantly 
faster and scales much better, and the promising experimental 
results on  several real-world datasets have further validated the scalability 
and extensibility of the SEP-QN framework.

\clearpage
\bibliographystyle{aaai}
\bibliography{reference}

\newpage
\appendixtitleon
\appendixtitletocon
\begin{appendices}

\setcounter{equation}{0}
\renewcommand\theequation{\Alph{section}.\arabic{equation}}
\newtheorem{athm}{Theorem}[section]
\newtheorem{alem}[athm]{Lemma}

\section{}
\subsection{Proof of Theorem \ref{th:unit_step_sufficient}}
\begin{alem} \label{lm:small_than_H}
If ~$ \MH_k $ is positive definite, then $ \Delta \Vx_k $ satisfies
$$ \nabla g(\Vx_k)^T \Delta \Vx_k + \Psi(\Vx_k + \Delta \Vx_k) - \Psi(\Vx_k) 
\leq -\Delta \Vx_k^T \MH_k \Delta \Vx_k $$ \end{alem}
The proof of this lemma is shown in \cite{lee2014proximal}.

\stepsufficient*
\begin{proof}
By Lemma \ref{lm:small_than_H}, we have
{\small
\begin{align}
 & \nabla g(\Vx_k)^T \Delta \Vx_k + \Psi(\Vx_k + \Delta \Vx_k) - 
\Psi(\Vx_k) + \frac{1}{2}\Delta \Vx_k^T  \MH_k \Delta \Vx_k \nonumber \\
& \leq -\frac{1}{2}\Delta \Vx_k^T \MH_k \Delta \Vx_k \leq 0. \nonumber
 \end{align}}
Since $ \gamma_k = \nabla g(\Vx_k)^T\Delta \Vx_k  + \Psi(\Vx_k+\Delta \Vx_k) - 
\Psi(\Vx_k) $ (the sufficient descent condition
\eqref{eq:sufficient_descent_condition}),
we have
 {\small\begin{align} \label{eq:T2B1}
 & \nabla g(\Vx_k)^T \Delta \Vx_k + \Psi(\Vx_k + \Delta \Vx_k) - \Psi(\Vx_k)
 + \frac{1-\alpha}{2}\Delta \Vx_k^T \MH_k \Delta \Vx_k \nonumber \\
 &\leq \alpha \gamma_k  
\nonumber \\ & g(\Vx_k) + \nabla g(\Vx_k)^T \Delta \Vx_k + \Psi(\Vx_k + \Delta 
\Vx_k) + \frac{1-\alpha}{2}\Delta \Vx_k^T  \MH_k \Delta \Vx_k \nonumber \\ 
&\leq g(\Vx_k) + \Psi(\Vx_k) + \alpha  \gamma_k.
\end{align}
}
Since $ \nabla^2 g $ is Lipschitz continuous with constant $ L_2 $, the smooth 
part $ g(x)$ can be expanded in Taylor's series as following
{\small\begin{align} \label{eq:T2B2}
 f(\Vx_k  + \Delta \Vx_k) & =  g(\Vx_k) + \nabla g(\Vx_k)^T \Delta \Vx_k
 + \frac{1}{2}\Delta \Vx_k^T \nabla^2 g(\Vx_k) \Delta \Vx_k \nonumber \\
 & +o(\Delta \Vx_k^2)+ \Psi(\Vx_k + \Delta \Vx_k) \nonumber \\
& \leq  g(\Vx_k) + \nabla g(\Vx_k)^T \Delta \Vx_k
 + \frac{1}{2}\Delta \Vx_k^T \nabla^2 g(\Vx_k) \Delta \Vx_k \nonumber \\
& +\frac{L_2}{6}\| \Delta \Vx_k\|_2^3 + \Psi(\Vx_k + \Delta \Vx_k).
\end{align}}
Since $ \MH_k \succcurlyeq m\MI$, from Lemma \ref{lm:small_than_H}, we have
\begin{equation} \label{eq:T2B3}
 \| \Delta \Vx_k \|_2^2 \leq -\frac{\gamma_k}{m}.
\end{equation}

Because $  (1-\alpha)\MH_k \succ \nabla^2 g(\Vx_k) $, we use the results in
\eqref{eq:T2B1}, \eqref{eq:T2B2} and \eqref{eq:T2B3} to yield
\begin{align}
 f( \Vx_k+\Delta \Vx_k) & \leq g(\Vx_k) + \nabla g(\Vx_k)^T \Delta \Vx_k +
\Psi(\Vx_k + \Delta \Vx_k) + \nonumber \\
&\frac{1-\alpha}{2}\Delta \Vx_k^T \MH_k \Delta
\Vx_k + \frac{L_2}{6}\| \Delta \Vx_k\|_2^3 \nonumber \\
& \leq f(\Vx_k) + \alpha  \gamma_k + \frac{L_2}{6}\| \Delta \Vx_k\|_2^3 
\nonumber \\
& \leq f(\Vx_k) + \alpha  \gamma_k - \frac{L_2}{6m}\| \Delta \Vx_k\|_2 \gamma_k.
\nonumber
\end{align}
We can show that $ \|\Delta \Vx_k\|_2 $ converges to zero via Theorem
\ref{th:gloal_c}. Hence, for $k$ sufficiently large, the unit
step length satisfies the sufficient descent condition
\eqref{eq:sufficient_descent_condition}.
\end{proof}

\subsection{Proof of Theorem \ref{th:larger_H0}}
\largeH*
\begin{proof}
By assumptions $ \Vs_k^T\Vy_k > 0 $ and $ H_0 \succ \MZERO $, we can prove this
result using the BFGS updating formula,
\[
\MH_{k+1}^{-1} = (\MI - \frac{\Vs_k \Vy_k^T}{\Vy_k^T \Vs_k}) \MH_k^{-1} (\MI -
\frac{\Vy_k \Vs_k^T}{\Vy_k^T \Vs_k}) + \frac{\Vs_k \Vs_k^T}{\Vy_k^T \Vs_k}.
\]
By reduction, if $ \MH_k^a \succ \MH_k^b \succ \MZERO $, $$ (\MH_{k+1}^b)^{-1} 
- 
(\MH_{k+1}^a)^{-1} =  (\MI - \frac{\Vs_k \Vy_k^T}{\Vy_k^T \Vs_k}) 
((\MH_k^b)^{-1} - (\MH_k^a)^{-1}) (\MI - \frac{\Vy_k \Vs_k^T}{\Vy_k^T \Vs_k}) 
$$ 
is positive definite when $ \Vs_k^T \Vy_k > 0$, then $ \MH_{k+1}^a \succ 
\MH_{k+1}^b   $, so with
larger initial $ \MH_0 $, we could obtain larger $ \MH_k $.
\end{proof}

\subsection{Proof of Theorem \ref{th:gloal_c}}
\begin{alem} \label{alm:global_c}
Suppose $ f $ is a closed convex function and $ \underset{\Vx} {\inf}\{f(\Vx) | 
\Vx \in \dom f\} $ is attained at some $\Vx^*$, If $ \MH_k \succ l\MI$ for some 
$ l > 0 $ and the surrogate problem \eqref{eq:proximal_local_model} is solved 
exactly in proximal quasi-Newton method, then $ \Vx_k $ converges to an optimal 
solution starting at any $ \Vx_0 \in \dom f$.
\end{alem}
The proof of this Lemma is shown in \cite{lee2014proximal}.
\globalc*
\begin{proof}
From Algorithm \ref{alg:sub_scd_solver}, SEP-QN use SCD to solve the local
proximal of the composite functions, and the adaptive Hessian strategy
keep $ \MH_k \succ l\MI $~ for some $ l > 0 $. By continuation SCD, the 
surrogate problem \eqref{eq:proximal_local_model}
would be solved exactly. Based on Lemma \ref{alm:global_c} and Assumption 
\ref{as:exist_solution}, $ \Vx_k $ converges to an optimal solution $ \Vx^* $ 
starting at any $ \Vx_0 \in \dom f$.
\end{proof}

\subsection{Proof of Theorem \ref{th:local_c}}
\localc*
\begin{proof}
 After sufficiently many iterations, $ \beta \approx 1 $. Under the Dennis-More 
criterion, we can show that the unit step length satisfied the sufficient 
descent condition \eqref{eq:sufficient_descent_condition} via the argument used 
in the proof of Theorem \ref{th:unit_step_sufficient}.
 Then we have,
\begin{align} \label{eq:less_newton}
 \| Vx_{k+1} - \Vx^* \|_2 & = \| \Vx_k + \Delta \Vx_k - \Vx^* \|_2 
\nonumber \\
& = \| \Vx_k + \Delta \Vx_k + \Delta \Vx_k^{nt} - \Delta \Vx_k^{nt} - \Vx^* 
\|_2 \nonumber \\
& \leq \| \Vx_k + \Delta \Vx_k^{nt} - \Vx^* \|_2 + \| \Delta \Vx_k  - \Delta 
\Vx_k^{nt} \|_2
\end{align}
According to Theorem 3.4 in \cite{lee2014proximal}, the proximal Newton method 
convergence quadraticly, that is,
\begin{align} \label{eq:total_ie}
 \| \Vx_{k+1}^{nt} - \Vx^* \|_2 & \leq \frac{L_2}{2l} \| \Vx_k^{nt} - \Vx^* 
\|_2^2
\end{align}
By continuation SCD, the surrogate problem 
\eqref{eq:proximal_local_model} would be solved exactly. We draw the same 
conclusion as \cite{lee2014proximal}, that is,
\begin{align} \label{eq:delta_ie}
 & \| \Delta \Vx_k  - \Delta \Vx_k^{nt} \|_2 \leq c_1  \|  \Vx_k - \Vx^* 
\|_2^{\frac{1}{2}}\| \Delta \Vx_k \|_2 + o(\| \Delta \Vx_k  \|_2) \nonumber \\
& \| \Delta \Vx_k  \|_2  \leq c_2 \|  \Delta \Vx^{nt} \|_2  = c_2 
\|\Vx_{k+1}^{nt} - \Vx_k \|_2 \nonumber \\
&\leq O(\| \Vx_{k} - \Vx^* \|_2^2) +c_2 ( \|\Vx_k - \Vx^* \|_2)
\end{align}
From \eqref{eq:less_newton}, \eqref{eq:total_ie} and \eqref{eq:delta_ie}, we 
conclude that,
\begin{align}
  \| \Vx_{k+1} - \Vx^* \|_2 & \leq \frac{L_2}{2l} \| \Vx_k^{nt} - \Vx^* \|_2^2  
+ o( \| \Vx_k - \Vx^* \|_2). \nonumber
\end{align}
Because proximal Newton method convergence much quickly, we deduce that $ \Vx_k 
$ converges to $ \Vx^* $ superlinearly.
\end{proof}
\end{appendices}

\end{document}